\documentclass[a4paper,11pt]{amsart}
\usepackage{amsfonts, amssymb, amsmath}
\usepackage{mathtools}

\usepackage{enumerate}
\usepackage{esint}
\usepackage{latexsym,amssymb}
\usepackage{citeref}
\newtheorem{thm}{Theorem}[section]

\newtheorem{lem}[thm]{Lemma}

\makeatletter
\DeclareRobustCommand\widecheck[1]{{\mathpalette\@widecheck{#1}}}
\def\@widecheck#1#2{%
	\setbox\z@\hbox{\m@th$#1#2$}%
	\setbox\tw@\hbox{\m@th$#1%
		\widehat{%
			\vrule\@width\z@\@height\ht\z@
			\vrule\@height\z@\@width\wd\z@}$}%
	\dp\tw@-\ht\z@
	\@tempdima\ht\z@ \advance\@tempdima2\ht\tw@ \divide\@tempdima\thr@@
	\setbox\tw@\hbox{%
		\raise\@tempdima\hbox{\scalebox{1}[-1]{\lower\@tempdima\box
				\tw@}}}%
	{\ooalign{\box\tw@ \cr \box\z@}}}
\makeatother

\newcommand{\lesi}{\lesssim}

\newcommand{\supp}{\operatorname{supp}}

\newcommand{\f}{\frac}

\newcommand{\dm}{d\mu}

\newcommand{\vc}{\infty}

\newcommand{\tx}{\tau^x}
\newcommand{\ty}{\tau^y}


\begin{document}
\title[Imaginary powers of Bessel operators]
{Sharp estimates for imaginary powers\\ of Bessel operators}

\authors

\author{The Anh Bui}
\address{School of Mathematical and Physical Sciences, Macquarie University, NSW 2109,
		Australia}
	
	\email{the.bui@mq.edu.au}

		\author{Xuejing Huo}
		\address{School of Mathematical and Physical Sciences, Macquarie University, NSW 2109,
				Australia}

		\email{xuejing.huo@students.mq.edu.au}

	\author{Ji Li}
	\address{School of Mathematical and Physical Sciences, Macquarie University, NSW 2109,
			Australia}
		
		\email{ji.li@mq.edu.au}

\date{}

\maketitle

\begin{abstract}
Let $L f(x):=-\frac{d^2}{dx^2}f(x)-\frac{ r}{x}\frac{d}{dx}f(x),\quad x>0, r>0$ be the Bessel operator on $((0,\vc), |\cdot|, x^rdx)$. In this paper, we prove the sharp weak type $(1,1)$ estimate for the imaginary power $L^{i\alpha}, \alpha\in \mathbb R$, of the Bessel operator.
\end{abstract}

\section{Introduction}

Let $ r > 0$. Consider the space $X=((0,\vc), |\cdot|, d\mu(x))$ where the measure $\dm$ defined by  $d\mu(x) = x^r dx$, and $|\cdot|$ is the usual Euclidean distance. Denote by $I_a(x)=\{x\in X: |x-y|<a\}$ the interval centered $x\in X$ with the length of $2a>0$. It is easy to see that
\begin{equation}
	\label{eq- volume}
	\mu(I_{a}(x))\simeq\begin{cases}
		ax^r, \ \ \ &x>a\\
		a^{r+1}, \ \ \ & 0<x<a. 
	\end{cases}
\end{equation}
Setting $n=r+1$, it follows that there exists $C>0$ such that 
\begin{equation}\label{doublingpro1}
	\mu(I_{2a}(x))\le C a^{n}\mu(I_{a}(x))
\end{equation}
for all $x\in X$ and $a>0$. Therefore, the space $((0,\vc), |\cdot|, \mu)$ is a space of homogeneous type in the sense on Coifmann and Weiss (see \cite{CW}).

In this paper, we consider the second order Bessel differential operator 
\begin{equation}\label{bessel 1}
	L f(x):=-\frac{d^2}{dx^2}f(x)-\frac{ r}{x}\frac{d}{dx}f(x),\quad x>0,
\end{equation}
 studied by Muckenhoupt--Stein \cite{MS}. 
The harmonic analysis related to the Bessel operator is an interesting topic and has attracted a great deal of attention. See for example \cite{BCC, DPW, GS, S1, W} and the references therein.

By the spectral theory, $L$ admits a spectral resolution
\begin{equation*}
	L = \int_{0}^{\infty}\lambda dE(\lambda),
\end{equation*}
where $\{E(\lambda): \lambda \geq 0\}$ is the spectral resolution of $L$.
If $F$ is a bounded Borel measurable function
on $[0,\infty)$, then the operator
\begin{align*}
	F(L) = \int_{0}^{\infty} F(\lambda)dE(\lambda)
\end{align*}
is bounded on $L^2 (X)$. In what follows, we use $K_{F(L)}(x,y)$ to denote
the kernel of $F(L)$.

In this paper, we consider the imaginary power operator $L^{i\alpha}$ with $\alpha\in \mathbb R$.
By the spectral theory, we have $\|L^{i\alpha}\|_{L^2(X)\to L^2(X)}=1$. For the boundedness of
$L^{i\alpha}$ on $L^p(X)$ with $p\ne 2$, if we apply the general spectral multiplier theorems in \cite{DOS} (see also \cite{GS}) we have, for any $\epsilon>0$,
\begin{equation}\label{eq- not sharp bound}
\|L^{i\alpha}\|_{L^p(X)\to L^p(X)}\le C_\epsilon (1+|\alpha|)^{n|1/p- 1/2|+\epsilon}  \ \ \ \ \text{for $1<p<\infty$}.
\end{equation}
Note that in the classical case of the Laplacian the bound $(1+|\alpha|)^{n|1/p- 1/2|+\epsilon}$ in \eqref{eq- not sharp bound} can be replaced by $(1+|\alpha|)^{n|1/p- 1/2|}$.

The main aim of this paper is to prove the following sharp estimates for the imaginary power operator $L^{i\alpha}$.
\begin{thm}\label{main thm}
	Let $L$ be the Bessel operator defined by \eqref{bessel 1} with $r>0$. Then for each $\alpha\in \mathbb R$ we have
	\begin{equation}
		\|L^{i\alpha}\|_{L^1(X)\to L^{1,\vc}(X)}\lesi (1+|\alpha|)^{n/2}.
	\end{equation}
By interpolation, for $\alpha\in \mathbb R$ and $1<p<\vc$ we have
\begin{equation}
	\|L^{i\alpha}\|_{L^p(X)\to L^p(X)}\lesi (1+|\alpha|)^{n|1/p- 1/2|}.
\end{equation}
\end{thm}

In \cite{CW} such a sharp estimate was proved for a non-negative self-adjoint operator satisfying the Gaussian upper bound of order $2$, but the underlying space is required to satisfy the polynomial volume growth. Since the measure $d\mu$ in the Bessel setting does not satisfy the polynomial volume growth in \cite{CW}, Theorem \ref{main thm} is not a consequence of that in \cite{SW}.

\medskip

The organization of the paper is as follows. In Section 2, we recall some basic properties on the functional calculus of $L$. The proof of Theorem \ref{main thm} will be given in Section 3.

\bigskip

{\bf Notation.} Throughout this paper, we use $C$ to denote positive constants, which are independent of the main
parameters involved and whose values may vary at every occurrence.
By writing $f \lesssim g$, we mean that $f \leq Cg$. We also use
$f \sim g$ to denote that $C^{-1}g \leq f \leq C g$.

Let $I$ be an interval in $X$. If we do not specify anything, this means that $I = I_{r_I}(x_I)$. For each $\lambda>0$ and each interval $I$, we denote $\lambda I = I_{\lambda r_I}(x_I)$. For $j\in \mathbb N$ and an interval $I\subset X$, we will write
\[
S_j(I)=2^{j}\backslash 2^{j-1}I \ \ \ j\ge 1,
\]
and set $S_0(I)=I$.

\section{Preliminaries}

Note that since the Bessel operator $L$ is a non-negative self-adjoint operator, it generates a semigroup $e^{-tL}$ for $t>0$. Moreover, the kernel $p_t(x,y)$ of the semigroup $e^{-tL}$ satisfies the Gaussian upper bound, i.e., there exist $C,c>0$ such that
\begin{equation*}
	\big|p_t(x,y)\big|\le  \f{C}{\mu(I_{\sqrt t}(x))}\exp\Big(-\f{|x-y|^{2}}{ct}\Big)
\end{equation*}
for all $x,y\in X$ and $t>0$.

As a consequence of \cite[Lemma 7]{Sikora} (see also \cite[Lemma 2]{SW}), we have:
\begin{lem}\label{lem:finite propagation}
	Let $\varphi\in C^\vc_0(\mathbb{R})$ be an even function with {\rm supp}\,$\varphi\subset (-1, 1)$ and $\displaystyle \int \varphi =2\pi$. Denote by $\Phi$ the Fourier transform of $\varphi$.  Then  the kernel $K_{\Phi(t\sqrt{L})}$ of $\Phi(t\sqrt{L})$ satisfies 
	\begin{equation}\label{eq1-lemPsiL}
		\displaystyle
		{\rm supp}\,K_{\Phi(t\sqrt{L})}\subset \{(x,y)\in X\times X:
		|x-y|\leq t\},
	\end{equation}
	and
	\begin{equation}\label{eq2-lemPsiL}
		|K_{\Phi(t\sqrt{L})}(x,y)|\lesi  \f{1}{\mu(I_t(x))}
	\end{equation}
	for all $x,y \in X$ and $t>0$.
\end{lem}

It is interesting to note that similar to the classical case, the spectral multiplier of $L$ can be defined by  the Fourier-Bessel transform. We first recall the concept of the Fourier-Bessel transform. See for example \cite{GS,S1}. Let $\widehat f(\lambda), \lambda>0$, denote the Fourier-Bessel transform of the function $f\in L^1(X)$. That is,
\begin{equation}
	\label{eq-FB expression}
	\widehat f(\lambda)=\int_0^\vc f(x)\phi_\lambda(x)	d\mu(x),
\end{equation}
where $\phi_\lambda(x)=a( r)(\lambda x)^{-( r-1)/2}J_{( r-1)/2}(x)$, $x\ge 0$, $a( r)=2^{( r-1)/2}\Gamma(( r+1)/2)$, and $J_{\nu}$ denotes the Bessel function of the first kind of order $\nu$. The functions $\phi_\lambda, \lambda>0$  are eigenfunctions of the Bessel operator $L$. It is well-known that
\[
L\phi_\lambda = \lambda^2 \phi_\lambda, \ \ \ \lambda>0.
\]
See for example \cite{L,T}.

We also have
\begin{equation}
	\label{eq-FB expression-inverse}
	f(x)=a( r)^{-1}\int_0^\vc \widehat f(\lambda)\phi_\lambda(x)	d\mu(\lambda),
\end{equation}
almost everywhere providing $f, \widehat f\in L^1(X)$.

The following Plancherel's formula holds true:
\begin{equation}\label{eq-Plancherel}
\|f\|_2 = a( r)^{-1}\|\widehat f(\lambda)\|_2.
\end{equation}
For any bounded function $m$ on $(0,\vc)$, we now define the Fourier-Bessel multiplier operator 
\[
(T_mf)\ \widehat{} = m\widehat f.
\]
By the spectral theory, we have
\begin{equation}\label{eq-spectral multiplier and BF transforms}
m(\sqrt L)f = T_{ m}f.
\end{equation}

For $y\in X$, the generalized translation $\ty$ is given by
$$
\ty f(x)=\int_{|x-y|}^{x+y}f(z)dW_{x,y}(z),
$$
where $\displaystyle dW_{x,y}(z)=dW_{x,y}(z)$ and $dW_{x,y}(z)$ is the one-dimension probability measure supported in the interval $[|x-y|, x+y]$ and given by
$$
dW_{x,y}(z)=c(r) {{ \Delta (x,y,z)^{r-2}}\over{(xyz)^{r-1}}} \ d\mu(z),
$$
where $c(r)=2^{r-2}\Gamma((r+1)/2)\Gamma(r/2)^{-1}\pi^{-1/2}$ and $\Delta(x,y,z)$ denotes the area of a triangle with three side lengths $x,y,z$. It is known that $\ty$ is a contraction operator on $L^p(X)$ for all $1\leq p\leq \vc$. More precisely, we have
\begin{equation}
	\label{eq-contraction}
	\|\tau^y f\|_{L^p(X)}\leq \|f\|_{L^p(X)}
\end{equation}
for all $f\in L^p(X), p\in [1,\vc]$ and $y\in X$.

We now define the generalized convolution of two appropriate functions $f$ and $g$ by setting
$$
f\ast g(x)=\int_X \tx f(y)g(y)d\mu(y)=\int_X \ty f(x)g(y)d\mu(y).
$$
It is obvious that $f\ast g(x) = g\ast f(x)$ and 
\[(f\ast g)^{\widehat{}}\ (x)=\widehat{f}(x)\widehat{g}(x). 
\]

\section{Sharp estimates for the imaginary powers $L^{i\alpha}$}

This section is dedicated to proving Theorem \ref{main thm}.
\begin{proof}[Proof of Theorem \ref{main thm}:]
		We follow the standard strategy  as in \cite{F} (see also \cite{He, DM, CW}).  
		
		We need to prove that 
		\[
		\mu\big(\big\{|L^{i\alpha}f|>\lambda\big\}\big)\lesi (1+|\alpha|)^{n/2}\f{\|f\|_{L^1(X)}}{\lambda}
		\]
		for all $f\in L^1(X)$ and $\lambda >0$.
		
		Fix $f\in L^1(X)$ and $\lambda>0$. By the Calder\'on-Zygmund decomposition,  we can decompose $f =g+\sum_k b_k=: g +b$ such that the good part $g$ and the bad bad $b$ satisfy the following conditions:
	\begin{equation}\label{eq-h}
		| g(x)|\lesssim \lambda \ \text{for a.e. $x\in X$}, \ \ \ \|g\|_{L^1(X)}\lesssim \|f\|_{L^1(X)},
	\end{equation}
	and
	\begin{enumerate}[(i)]
		\item $\supp b_k \subset I_k$ for some interval $I_k\subset X$ and for each $k$;\\
		
		\item $\|b_k\|_{L^1(X)}\le \lambda  \mu(I_k)$;\\
		
		\item $\sum_{k}\mu(I_k)\lesssim \frac{\|f\|_{L^1(X)}}{\lambda}$;\\
		
		\item $\sum_{k}\chi_{2I_k}\lesssim 1$.
	\end{enumerate}
	See for example \cite{CW}.
	
	As usual, we write
	\[
	\mu\Big(\Big\{|L^{i\alpha}f|>\lambda\Big\}\Big)\le \mu\Big(\Big\{|L^{i\alpha}g|>\lambda/2\Big\}\Big) +\mu\Big(\Big\{|L^{i\alpha}b|>\lambda/2\Big\}\Big).
	\]
	The term related to the good part can be estimated by the standard argument. Using Chebyshev's inequality and the $L^2$-boundedness of $L^{i\alpha}$,
	\[
	\begin{aligned}
		\mu\Big(\Big\{|L^{i\alpha}g|>\lambda/2\Big\}\Big)&\lesssim \f{\|L^{i\alpha}g\|^2_2}{\lambda^2}\\
		&\lesssim \f{\|g\|^2_2}{\lambda^2}\\
		&\lesssim \f{\lambda \|g\|_{L^1(X)} }{\lambda^2}\\
		&\lesssim \f{\|f\|_{L^1(X)}}{\lambda}.
	\end{aligned}
	\]
	
	For the bad part, define 
	\[
	\theta = \f{1}{4M \sqrt{1+|\alpha|}},
	\]
	where $M$ is an positive integer which will be fixed later and let $\Phi$ be the function in Lemma \ref{lem:finite propagation}. We then set $\Phi(\theta r_{I_k} t)=\Phi(\theta r_{I_k}t)$. We now have
	\[
	\begin{aligned}
		\mu\Big(\Big\{|L^{i\alpha}b|>\lambda/2\Big\}\Big)&\le  \mu\Big(\Big\{\Big|L^{i\alpha}\Big[\sum_{k} \Big(I-\big(I-\Phi(\theta r_{I_k}\sqrt L)\big)^M\Big)b_k\Big]\Big|>\lambda/4\Big\}\Big)\\
		& \ \ + \mu\Big(\Big\{\Big|\sum_{k} L^{i\alpha}\big(I-\Phi(\theta r_{I_k}\sqrt L)\big)^M b_k\Big|>\lambda/4\Big\}\Big)\\
		&=: E_1 + E_2.
	\end{aligned}
	\]
	In order to estimate the term $E_1$, we note that
	\[
	\Psi(\theta r_{I_k}\sqrt L):=I-\big(I-\Phi(\theta r_{I_k}\sqrt L)\big)^M = \sum_{k=1}^M c_k[\Phi(\theta r_{I_k}\sqrt L)]^k,
	\]	
	where $c_k$ are coefficients.
	
	From Lemma \ref{lem:finite propagation}, 
	\[
	K_{\Psi(\theta r_{I_k}\sqrt L)}(\cdot,\cdot)\subset \{(x,y): d(x,y)<r_{I_k}/2\},
	\]
	which implies
	\begin{equation}\label{eq1- bad part}
		\Psi(\theta r_{I_k}\sqrt L)b_k\subset 2I_k,
	\end{equation}
	and
	\begin{equation}\label{eq2- bad part}
		\begin{aligned}
		\sup_{x\in 2I_k}|K_{\Psi(\theta r_{I_k}\sqrt L)}(x,y)|&\lesi  \sup_{x\in 2I_k}\f{1}{\mu(I_{\theta r_{I_k}}(x))}\\
		&\lesi  \sup_{x\in 2I_k}\f{\theta^{-n}}{\mu(I_{r_{I_k}}(x))}\simeq  \f{\theta^{-n}}{\mu(I_k)},
		\end{aligned}
	\end{equation}
where we used \eqref{doublingpro1} in the second inequality.

It follows that 
\[
\begin{aligned}
	\Big\|\sum_k\Psi(\theta r_{I_k}\sqrt L)b_k\Big\|_\vc&\lesi \sum_k\|\Psi(\theta r_{I_k}\sqrt L)b_k\|_\vc\\
	&\lesi \sum_k \f{\theta^{-n}}{\mu(I_k)}\|b_k\|_{L^1(X)}.1_{2I_k}\\
	&\lesi \lambda\theta^{-n},
\end{aligned}
\]
and
\[
\begin{aligned}
	\Big\|\sum_k\Psi(\theta r_{I_k}\sqrt L)b_k\Big\|_{L^1(X)}&\lesi \sum_k\|\Psi(\theta r_{I_k}\sqrt L)b_k\|_{L^1(X)}\\
	&\lesi \sum_k \|b_k\|_{L^1(X)}\\
	&\lesi \|f\|_{L^1(X)},
\end{aligned}
\]
where in the second inequality we used the $L^1$-boundedness of $\Psi(\theta r_{I_k}\sqrt L)$, which is followed from Lemma \ref{lem:finite propagation}.

Therefore, by the Chebyshev inequality, the $L^2$-boundedness of $L^{i\alpha}$,  (iv), \eqref{eq1- bad part}, \eqref{eq2- bad part} and (ii), we have
	\[
	\begin{aligned}
		E_1& \lesi \f{\big\|\sum_k\Psi(\theta r_{I_k}\sqrt L)b_k\big\|_2^2}{\lambda^2}\\
		&\lesi \f{\big\|\sum_k\Psi(\theta r_{I_k}\sqrt L)b_k\big\|_2^2}{\lambda^2}\\
		&\lesi \f{\lambda \theta^{-n}\big\|\sum_k\Psi(\theta r_{I_k}\sqrt L)b_k\big\|_{L^1(X)}}{\lambda^2}\\
		&\lesi \f{\theta^{-n} \|f\|_{L^1(X)}}{\lambda}\\
		&\simeq (1+|\alpha|)^{n/2}\f{\|f\|_{L^1(X)}}{\lambda}.
	\end{aligned}
	\]

	It remains to estimate $E_2$. To do this, we write
	\[
	\begin{aligned}
		E_2&\le \mu\Big(\bigcup_k 4I_k^*\Big) +  \mu\Big(\Big\{x \notin \bigcup_k 4I_k^*: \Big|\sum_{k} L^{i\alpha}\big(I-\Phi(\theta r_{I_k}\sqrt L)\big)^Mb_k(x)\Big|>\lambda/4\Big\}\Big)\\
		&=:E_{21}+E_{22},
	\end{aligned}
	\]
	where $I_k^* = \sigma I_k$ with $\sigma =\sqrt{1+|\alpha|}$.
	
	By \eqref{doublingpro1} and (iii),
	\[
	\begin{aligned}
		E_{22}&\le \sum_k\mu(4I_k^*)\\
		&\lesi \sigma^{n}\sum_k\mu(I_k)\\
		&\lesi (1+|\alpha|)^{n/2}\f{\|f\|_{L^1(X)}}{\lambda}.	
	\end{aligned}
	\]
	For the term $E_{22}$, by the Chebyshev inequality,
	\[
	\begin{aligned}
		E_{22}&\lesi \f{\Big\|\sum_{k} L^{i\alpha}\big(I-\Phi(\theta r_{I_k}\sqrt L)\big)^Mb_k\Big\|_{L^1(X\backslash \cup_k 4I^*_k)}}{\lambda}\\
		&\lesi \f{\sum_{k}\| L^{i\alpha}\big(I-\Phi(\theta r_{I_k}\sqrt L)\big)^Mb_k\|_{L^1(X\backslash 4I^*_k)}}{\lambda}.
	\end{aligned}	
	\]
		It suffices to prove that for each $k$,
	\begin{equation}\label{eq-key estimate}
	\sup_{y\in I_k}\int_{X\backslash 4I^*_k} |K_{L^{i\alpha}\big(I-\Phi(\theta r_{I_k}\sqrt L)\big)^M}(x,y)|d\mu(x)\lesi (1+|\alpha|)^{n/2},
	\end{equation}
where $K_{L^{i\alpha}\big(I-\Phi(\theta r_{I_k}\sqrt L)\big)^M}(x,y)$ is the kernel of $L^{i\alpha}\big(I-\Phi(\theta r_{I_k}\sqrt L)\big)^M$.

	Once \eqref{eq-key estimate} has been proved, we get that
	\[
	\begin{aligned}
		E_{22}&\lesi (1+|\alpha|)^{n/2}\f{\sum_k \|b_k\|_{L^1(X)}}{\lambda}\\
		&\lesi (1+|\alpha|)^{n/2}\f{ \|f\|_{L^1(X)}}{\lambda},
	\end{aligned}
	\]
	and this completes our proof.
	
	We now prove the claim \eqref{eq-key estimate}. To do this, let $\psi\in C^\vc_c(\mathbb R)$ be a even function supported in $\{\xi: 1/4\le |\xi|\le 4\}$ and $\psi=1$ on $\{\xi: 1/2\le |\xi|\le 2\}$ such that 
	\[
	\sum_{\ell \in \mathbb Z}\psi(2^{-\ell}x)=1, \ \ \ \ \ x\ne 0.
	\]
	Set
	\[
	F_{\ell,\theta,r_{I_k}}(x)= \psi(2^{-\ell}x) x^{2i\alpha}(1-\Phi_{\theta,r_{I_k}}(x))^M 	\]
	so that 
	\[
	 x^{2i\alpha}(1-\Phi_{\theta,r_{I_k}}(t))^M=\sum_{\ell\in \mathbb Z}F_{\ell,\theta,r_{I_k}}(x), \ \ \ x\ne 0.
	\]
	Then we have, for each $k$ and $y\in I_k$,
	\[
	\begin{aligned}
		\int_{X\backslash 4I^*_k} |K_{L^{i\alpha}\big(I-\Phi(\theta r_{I_k}\sqrt L)\big)^M}(x,y)|d\mu(x)&\le \sum_{\ell\in \mathbb Z}\int_{X\backslash 4I^*_k} |K_{F_{\ell,\theta,r_{I_k}}(\sqrt L)}(x,y)|d\mu(x),
	\end{aligned}
	\]
	where $K_{F_{\ell,\theta,r_{I_k}}(\sqrt L)}(x,y)$ is the kernel of $F_{\ell,\theta,r_{I_k}}(\sqrt L)$.

	According to \eqref{eq-spectral multiplier and BF transforms},
	\[
	F_{\ell,\theta,r_{I_k}}(\sqrt L)f = T_{F_{\ell,\theta,r_{I_k}}}f.
	\]
	It follows that  
	\[
	K_{F_{\ell,\theta,r_{I_k}}(\sqrt L)}(x,y)= {\tau}^yk_{\ell,\theta,r_{I_k}}(x),
	\]
	where
	\[	
	k_{\ell,\theta,r_{I_k}}(x)=\widecheck{F}_{{\ell,\theta,r_{I_k}}}(x):=a( r)^{-1}\int_0^\vc F_{\ell,\theta,r_{I_k}}(\lambda)\phi_\lambda(x)\lambda^ r d\lambda.
	\]
	Therefore,
	for each $k$ and $y\in I_k$,
	\[
	\begin{aligned}
		\int_{X\backslash 4I^*_k} |K_{L^{i\alpha}\big(I-\Phi(\theta r_{I_k}\sqrt L)\big)^M}(x,y)|d\mu(x)&\le \sum_{\ell\in \mathbb Z}\int_{X\backslash 4I^*_k} |{\tau}^yk_{\ell,\theta,r_{I_k}}(x)|d\mu(x).
	\end{aligned}
	\]
	Set $A_{k} = \{z: z>\sigma r_{I_k}\}$ for each $k$ and $y\in I_k$. Then for $y\in I_k$ and $x\in X\backslash 4I^*_k$ we have
	\begin{equation}\label{eq- ty is 0}
	{\tau}^y[k_{\ell,\theta,r_{I_k}}.1_{X\backslash A_{k}}](x)=0.
	\end{equation}
	Indeed, if $z\in 1_{X\backslash A_{k}}$, i.e., $z\le \sigma r_{I_k} $, then $z<|x-y|$ as long as $x\in X\backslash 4I_k^*$ and $ y\in I_k$. It follows \eqref{eq- ty is 0}. Consequently,
	for each $k$ and $y\in I_k$,
	\begin{equation}\label{eq-last eq}
	\begin{aligned}
		\int_{X\backslash 4I^*_k} |K_{L^{i\alpha}\big(I-\Phi(\theta r_{I_k}\sqrt L)\big)^M}(x,y)|d\mu(x)&\le \sum_{\ell\in \mathbb Z}\int_{X\backslash 4I^*_k} |{\tau}^y[k_{\ell,\theta,r_{I_k}}.1_{A_{k}}](x)|d\mu(x)\\
		&\le \sum_{\ell\in \mathbb Z} \int_{X}|[k_{\ell,\theta,r_{I_k}}.1_{A_{k}}](x)|d\mu(x)\\
		& =\sum_{\ell\in \mathbb Z}\int_{x> \sigma r_{I_k}}|k_{\ell,\theta,r_{I_k}}(x)|d\mu(x),
	\end{aligned}
	\end{equation}
	where in the second inequality we used \eqref{eq-contraction}.
	
	Using the Plancherel theorem \eqref{eq-Plancherel}, for  a fixed $s_0\in 2\mathbb N, s_0>n/2$,
	\[
	\begin{aligned}
		\|(1+(2^\ell x)^2)^{s_0/2}k_{\ell,\theta,r_{I_k}}\|_2 &= \|[(1+2^{2\ell}L)^{s_0/2}F_{\ell,\theta,r_{I_k}}]^{\ \widecheck{}}\ \|_2\\
		&=a(r)^{-1}\|(1+2^{2\ell}L)^{s_0/2}F_{\ell,\theta,r_{I_k}}\|_2.
	\end{aligned}
	\]
Using the form
\[
L=-\frac{d^2}{dx^2} -\frac{ r}{x}\frac{d}{dx},
\]
and the fact that  since $\supp F_{\ell,\theta,r_{I_k}} \subset [2^{\ell-1}, 2^{\ell+1}]$ with $F_{\ell,\theta,r_{I_k}}(x)= \psi(2^{-\ell}x) x^{2i\alpha}(1-\Phi_{\theta,r_{I_k}}(x))^M$, we have the following useful remarks:
\begin{itemize}
	\item If we apply either the operator $d/dx$ or $r/x$ for  the function $\psi(2^{-\ell}x)$ we gain a factor  $\simeq 2^\ell$;
	\item If we apply  the operator $d/dx$ and $r/x$ for  the function $x^{2i\alpha}$ we gain a factor $\simeq \alpha 2^{\ell}$ and $\simeq 2^\ell$, respectively. In both cases, the factor is majorized by $(1+|\alpha|)2^{\ell}$;
	\item Since $\Phi$ is an  even Schwartz function with $\Phi(0)=1$ and $\Phi'(0)=0$, $1-\Phi(t)\simeq t^2$ as $t\to 0$. Hence, if apply  the operator $d/dx$ and $r/x$ for $(1-\Phi_{\theta,r_{I_k}}(x))^M$ we will gain a factor which is less than or equal to (a multiple of) $\min\{1, (2^\ell\theta r_{B_k})^{2M}\}$. 
\end{itemize}
Taking these remarks into account, it can be verified that 
\[
\begin{aligned}
	\|(1+2^{2\ell}L)^{s_0/2}F_{\ell,\theta,r_{I_k}}\|_2& \lesi (1+|\alpha|)^{s_0}   \min\{1, (2^\ell \theta r_{I_k})^{2M}\} \Big(\int_{2^{\ell-1}}^{2^{\ell+1}}d\mu(x)\Big)^{1/2}\\
	&\lesi (1+|\alpha|)^{s_0} 2^{\ell n/2} \min\{1, (2^\ell \theta r_{I_k})^{2M}\},
\end{aligned} 
\]
which implies
\[
\|(1+(2^\ell x)^2)^{s_0/2}k_{\ell,\theta,r_{I_k}}\|_2\lesi (1+|\alpha|)^{s_0} 2^{\ell n/2} \min\{1, (2^\ell \theta  r_{I_k})^{2M}\}.
\]
This, together with H\"older's inequality, yields that for each $k, \ell$,
\[
\begin{aligned}
	&\int_{x> \sigma|y-x_{I_k}|}|k_{\ell,\theta,r_{I_k}}(x)|d\mu(x)\\
	&\le \|(1+(2^\ell x)^2)^{s_0/2}k_{\ell,\theta,r_{I_k}}\|_2 \Big(\int_{x> \sigma r_{I_k}} (1+(2^\ell x)^2)^{-s_0}d\mu(x)\Big)^{1/2}\\
	&\lesi (1+|\alpha|)^{s_0} 2^{\ell n/2} \min\{1, (2^\ell \theta r_{I_k})^{2M}\} \Big(\int_{x> \sigma r_{I_k}} (1+(2^\ell x)^2)^{-s_0}d\mu(x)\Big)^{1/2}.
\end{aligned}
\]
A straightforward calculation leads us to that 
\[
\Big(\int_{x> \sigma r_{I_k}} (1+(2^\ell x)^2)^{-s_0}d\mu(x)\Big)^{1/2}\le (\sigma r_{I_k})^{n/2} (2^\ell \sigma r_{I_k})^{-s_0}.
\]
Therefore,
\[
\begin{aligned}
	\int_{x> \sigma r_{I_k}}|k_{\ell,\theta,r_{I_k}}(x)|d\mu(x)& \lesi (1+|\alpha|)^{s_0} \min\{1, (2^\ell \theta  r_{I_k})^{2M}\} (2^\ell\sigma r_{I_k})^{n/2} (2^\ell \sigma r_{I_k})^{-s_0}\\
	& \lesi (1+|\alpha|)^{s_0} \min\{1, (2^\ell \theta  r_{I_k})^{2M}\}  (2^\ell \sigma r_{I_k})^{-(s_0-n/2)}\\
	& \lesi (1+|\alpha|)^{s_0} \min\{1, (2^\ell \theta  r_{I_k})^{2M}\}  (2^\ell \theta r_{I_k})^{-(s_0-n/2)}(\theta^{-1}\sigma)^{-(s_0-n/2)}\\
	& \lesi (1+|\alpha|)^{n/2} \min\{1, (2^\ell \theta  r_{I_k})^{2M}\}  (2^\ell \theta r_{I_k})^{-(s_0-n/2)},
\end{aligned}
\]
where in the last inequality we used the fact $\theta^{-1}\sigma =(1+|\alpha|)$.

Inserting into \eqref{eq-last eq}, we have
\[
\begin{aligned}
	\int_{X\backslash 4I^*_k} |K_{L^{i\alpha}\big(I-\Phi(\theta r_{I_k}\sqrt L)\big)^M}(x,y)|d\mu(x)&\lesi \sum_{\ell\in \mathbb Z} (1+|\alpha|)^{n/2} \min\{1, (2^\ell \theta  r_{I_k})^{2M}\}  (2^\ell \theta r_{I_k})^{-(s_0-n/2)}\\
	&\lesi (1+|\alpha|)^{n/2},
\end{aligned}
\]
as long as $2M>s_0-n/2$.

This completes our proof.
	\end{proof}

\bigskip
\textbf{Acknowledgement.}  T. A. Bui and J. Li were supported by the research grant ARC DP220100285 from the Australian Research Council.


\begin{thebibliography}{19}

\bibitem{BCC} J. Betancor, A. Castro and J. Curbelo, Spectral Multipliers for multidimensional Bessel operators,  {J. Fourier Anal. App.}  {17} (2011), 932--975.


\bibitem{CW} R.R. Coifman and G. Weiss, Extensions of Hardy spaces and their use in analysis, {\it Bull. Amer. Math. Soc.}  {83} (1977), 569--645.

 
\bibitem{DOS}  X.T. Duong, E.M. Ouhabaz  and A. Sikora,
 Plancherel-type estimates and sharp spectral multipliers.
{J. Funct. Anal.} {196} (2002),  443--485.

\bibitem{DM} X. T. Duong and A. McIntosh, Singular integral operators with non-smooth kernels on irregular domains, Rev. Mat. Iberoamericana 15 (1999), 233--265.


\bibitem{DPW} J. Dziuba\'nski, M. Preisner and B. Wr\'obel, Multivariate H\"ormander-type multiplier theorem for the Hankel transform, {J. Fourier Anal. App.} {19} (2013), 417--437.


\bibitem{F} C. Fefferman, Inequalities for strongly singular convolution operators, Acta Math. 124 (1970), 9--36.

\bibitem{GS} J. Gosselin and K. Stempak, A weak-type estimate for Fourier Bessel multipliers, {Proc. Am. Math.
Soc.} {106} (1989), 655--662.

\bibitem{He} W. Hebisch, Multiplier theorem on generalized Heisenberg groups, Colloq. Math. 65 (1993), 231--239.
	
\bibitem{L} N.N. Lebedev, \emph{Special Functions and Their applications}, Dover, New York, 1972.

\bibitem{MS} B. Muckenhoupt and  E. M. Stein,
Classical expansions and their relation to conjugate harmonic functions, Trans. Amer. Math. Soc. {118} (1965), 17--92.


\bibitem{S1} K. Stempak, La th\'eorie de Littlewood-Paley pour la transformation de Fourier--Bessel, {C. R. Acad. Sci. Paris} {303} (1986), 15--18.

\bibitem{T} E.C. Titchmarsh, \emph{Introduction to the Theory of Fourier Integrals}, Clareoton Press, Oxford, 1937.

\bibitem{Sikora} A. Sikora, Riesz transform, Gaussian bounds and the method of wave equation, Math. Z. 247 (2004), no. 3, 643--662.

\bibitem{SW}
A. Sikora and J. Wright, {Imaginary powers of Laplace operator}, Proc. Amer. Math. Soc. {129} (2001), 1745--1754.



\bibitem{W} G.N. Watson, \emph{A Treatise on the Theory of Bessel Functions}, Cambridge University Press, Cambridge
(1966).

\end{thebibliography}
\end{document}